\documentclass[12pt,lenq]{amsart}
\usepackage{amsmath}
\usepackage{amscd}
\usepackage{amssymb}
\usepackage{amsbsy}
\usepackage{amsfonts}
\usepackage{latexsym}
\usepackage{graphics}
\usepackage{amsmath,amscd,latexsym}
\usepackage{multirow}
\usepackage{array}

\theoremstyle{plain}
\newtheorem{theorem}{Theorem}[section]

\newtheorem{lemma}[theorem]{Lemma}
\newtheorem{corollary}[theorem]{Corollary}

\newtheorem{definition}[theorem]{Definition}

\newtheorem{algorithm}[theorem]{Algorithm}

\newlength\savewidth

\def\a{\alpha}
\def\b{\beta}
\def\d{\delta}

\def\s{\sigma}
\def\t{\tau}

\begin{document}

\title[On several problems about automorphisms
of the free group of rank two] {On several problems about
automorphisms of the free group of rank two}
\author{Donghi Lee}
\address{Department of Mathematics\\
Pusan National University \\
San-30 Jangjeon-Dong, Geumjung-Gu, Pusan, 609-735, Korea}
\email{donghi@pusan.ac.kr}


\begin{abstract} Let $F_n$ be a free group of rank $n$. In this paper
we discuss three algorithmic problems related to automorphisms
of $F_2$.

A word $u$ of $F_n$ is called positive if $u$ does not have
negative exponents. A word $u$ in $F_n$ is called potentially
positive if $\phi(u)$ is positive for some automorphism $\phi$ of
$F_n$. We prove that there is an algorithm to decide whether or
not a given word in $F_2$ is potentially positive, which gives an
affirmative solution to problem F34a in [1] for the case of $F_2$.

Two elements $u$ and $v$ in $F_n$ are said to be boundedly
translation equivalent if the ratio of the cyclic lengths of
$\phi(u)$ and $\phi(v)$ is bounded away from $0$ and from $\infty$
for every automorphism $\phi$ of $F_n$. We provide an algorithm to
determine whether or not two given elements of $F_2$ are boundedly
translation equivalent, thus answering question F38c in the online
version of [1] for the case of $F_2$.

We further prove that there exists an algorithm to decide whether
or not a given finitely generated subgroup of $F_2$ is the fixed
point group of some automorphism of $F_2$, which settles problem
F1b in [1] in the affirmative for the case of $F_2$.
\end{abstract}
\maketitle

\section{Introduction}

Let $F_n$ be the free group of rank $n \ge 2$ with basis $\Sigma$.
In particular, if $n=2$, we let
$\Sigma=\{a, b\}$, namely, $F_2$ is the free group with basis
$\{a, b\}$. A word $v$ in $F_n$ is called {\it cyclically reduced} if all its
cyclic permutations are reduced. A {\it cyclic word} is defined to
be the set of all cyclic permutations of a cyclically reduced
word. By $[v]$ we denote the cyclic word associated with a word
$v$. Also by $\|v\|$ we mean the length of the cyclic word $[v]$
associated with $v$, that is, the number of cyclic permutations of
a cyclically reduced word which is conjugate to $v$. The length
$\|v\|$ is called the {\it cyclic length} of ~$v$. For two
automorphisms $\phi$ and $\psi$ of $F_n$, by writing $\phi \equiv
\psi$ we mean the equality of $\phi$ and $\psi$ over all cyclic
words in $F_n$, that is, $\phi(w)=\psi(w)$ for every cyclic word
$w$ in $F_n$.

Recall that a {\it Whitehead automorphism} $\a$ of $F_n$ is
defined to be an automorphism of one of the following two types
(cf. \cite{lee}):

\begin{itemize}
\item[(W1)] $\a$ permutes elements in $\Sigma^{\pm 1}$.

\item[(W2)] $\a$ is defined by a letter $x \in \Sigma^{\pm 1}$ and
a set $S \subset \Sigma^{\pm 1} \setminus \{x, x^{-1}\}$ in such a
way that if $c \in \Sigma^{\pm 1}$ then (a) $\a(c)=cx$ provided $c
\in S$ and $c^{-1} \notin {S}$; (b) $\a(c)=x^{-1}cx$ provided both
$c,\, c^{-1} \in S$; (c) $\a(c)=c$ provided both $c,\, c^{-1}
\notin S$.
\end{itemize}
If $\a$ is of type (W2), we write $\a=(S, x)$. Note that in the expression of $\a=(S, x)$ it is conventional to
include the defining letter $x$ in the defining set $S$, but for the sake of brevity of notation we will
omit $a$ from $S$ as defined above.

Throughout the present paper, we let
\[
\s=(\{a\}, b), \quad \t=(\{b\}, a)
\]
be Whitehead automorphisms of type (W2) of $F_2$. Recently the
author ~\cite{lee} proved that every automorphism of $F_2$ can
represented in one of two particular types over all cyclic words
of $F_2$ as follows:

\begin{lemma} \label{lem:2.3} ([Lemma 2.3, 6])
For every automorphism $\phi$ of $F_2$, $\phi$ can be represented
as $\phi \equiv \b \phi'$, where $\b$ is a Whitehead automorphism
of $F_2$ of type (W1) and $\phi'$ is a chain of one of the forms
\[
\begin{aligned}
(C1) \ & \phi' \equiv \t^{m_k}\s^{l_k} \cdots \t^{m_1}\s^{l_1} \\
(C2) \ & \phi' \equiv \t^{-m_k}\s^{-l_k} \cdots \t^{-m_1}\s^{-l_1}
\end{aligned}
\]
with $k \in \mathbb {N}$ and both $l_i, \, m_i \ge 0$ for every
$i=1, \dots, k$.
\end{lemma}

With the notation of Lemma ~\ref{lem:2.3}, we define the {\it
length of an automorphism} $\phi$ of $F_2$ as $\sum_{i=1}^k
(m_i+l_i)$, which is denoted by $|\phi|$. Then obviously
$|\phi|=|\phi'|$.

In the present paper, with the help of Lemma ~\ref{lem:2.3}, we resolve three algorithmic
problems related to automorphisms of $F_2$. Indeed, the description of automorphisms $\phi$
of $F_2$ in the statement of Lemma ~\ref{lem:2.3} provides us with a very useful computational tool that
facilitates inductive arguments on $|\phi|$ in the proofs of the problems.

The first problem we deal with is about potential positivity of
elements in a free group the notion of which was first introduced
by Khan ~\cite{khan}.

\begin{definition}
A word $u$ of $F_n$ is called positive if $u$ does not have
negative exponents. A word $u$ in $F_n$ is called potentially
positive if $\phi(u)$ is positive for some automorphism $\phi$ of
$F_n$.
\end{definition}

It was shown by Khan ~\cite{khan} and independently by Meakin-Weil
~\cite{meakin} that the Hanna Neumann conjecture is satisfied if
one of the subgroups is generated by positive elements. 

In Section 2, we shall describe an algorithm to decide whether or
not a given word in $F_2$ is potentially positive, which gives an
affirmative solution to problem F34a in [1] for the case of $F_2$.

The second problem we discuss here is related to the notion of
bounded translation equivalence which is one of generalizations of the notion
of translation equivalence, due to
Kapovich-Levitt-Schupp-Shpilrain ~\cite{kapovich}.

\begin{definition} Two elements $u$ and $v$ in
$F_n$ are called translation equivalent in $F_n$ if
$\|\phi(u)\|=\|\phi(v)\|$ for every automorphism $\phi$ of $F_n$.
\end{definition}

Several different sources of translation equivalence in free
groups were provided by Kapovich-Levitt-Schupp-Shpilrain ~\cite{kapovich} and
the author ~\cite{lee2}. In another paper of the author ~\cite{lee}, it is proved that
there exists an algorithm to decide whether
or not two given elements $u$ and $v$ of $F_2$ are translation
equivalent. In contrast with the notion of translation
equivalence, bounded translation equivalence is defined as
follows:

\begin{definition}
Two elements $u$ and $v$ in $F_n$ are said to be boundedly
translation equivalent in $F_n$ if there is $C>0$ such that
\[{1 \over C} \le {\|\phi(u)\| \over \|\phi(v)\|} \le C
\]
for every automorphism $\phi$ of $F_2$.
\end{definition}

Clearly every pair of translation equivalent elements in $F_n$ are boundedly translation
equivalent in $F_n$, but not vice versa. As one of specific examples of volume equivalence,
we mention that two elements $a$ and $a[a, b]$ are boundedly translation equivalent in $F_2$.
Indeed, if $u=a$ and $v=a[a, b]$, then we have, in view of Lemma ~\ref{lem:2.3}, that
\[{1 \over 5} \le {\|\phi(u)\| \over \|\phi(v)\|} \le 1
\]
for every automorphism $\phi$ of $F_2$.

In Section 3, developing further the technique used in
~\cite{lee}, we shall demonstrate that there exists an algorithm
to determine whether or not two given elements of $F_2$ are
boundedly translation equivalent, thus affirmatively answering
question F38c in the online version of [1] for the case of $F_2$.

Our last problem is concerned with the notion of fixed point
groups of automorphisms of free groups.

\begin{definition} A subgroup $H$ of $F_n$ is called the fixed
point group of an automorphism $\phi$ of $F_n$ if $H$ is precisely
the set of the elements of $F_n$ which are fixed by $\phi$.
\end{definition}

Due to Bestvina-Handel ~\cite{bestvina}, a subgroup of rank bigger
than $n$ cannot possibly be the fixed point group of an
automorphism of $F_n$. Recently Martino-Ventura ~\cite{ventura}
provided an explicit description for the fixed point groups of
automorphisms of $F_n$, generalizing the maximal rank case studied
by Collins-Turner ~\cite{turner}. However, this description is not
a complete characterization of all fixed point groups of
automorphisms of $F_n$. On the other hand, Maslakova
~\cite{maslakova} proved that, given an automorphism $\phi$ of
$F_n$, it is possible to effectively find a finite set of
generators of the fixed point group of $\phi$.

In Section 4, we shall present an algorithm to decide whether or
not a given finitely generated subgroup of $F_2$ is the fixed
point group of some automorphism of $F_2$, which settles problem
F1b in ~\cite{baumslag} in the affirmative for the case of $F_2$.

\section{Potential positivity in $F_2$}

Recall that $F_2$ denotes the free group with basis $\Sigma=\{a,b\}$,
and that $\s$ and $\t$ denote Whitehead automorphisms
\[
\s=(\{a\}, b), \quad \t=(\{b\}, a)
\]
of $F_2$ of type (W2). We also recall from ~\cite{lee} the definition
of trivial or nontrivial cancellation. For a cyclic word $w$ in
$F_2$ and a Whitehead automorphism, say $\s$, of $F_2$, a subword
of the form $ab^ra^{-1}$ ($r \neq 0$), if any, in $w$ is invariant
in passing from $w$ to $\s(w)$, although there occurs cancellation
in $\s(ab^ra^{-1})$ (note that $\s(ab^ra^{-1})=ab \cdot b^r \cdot
b^{-1}a^{-1}=ab^ra^{-1}$). Such cancellation is called {\it
trivial cancellation}. And cancellation which is not trivial
cancellation is called {\it proper cancellation}. For example, a
subword $ab^{-r}a$ ($r \ge 1$), if any, in $w$ is transformed to
$ab^{-r+1}ab$ by applying $\s$, and thus the cancellation occurring in
$\s(ab^{-r}a)$ is proper cancellation.

The following lemma from ~\cite{lee} will play a fundamental role
throughout the present paper.

\begin{lemma} \label{lem:2.4} (Lemma 2.4 in \cite{lee}) Let $u$ be a cyclic word in $F_2$, and let
$\psi$ be a chain of type (C1) (or (C2)). If $\psi$ contains at
least $\|u\|$ factors of $\s$ (or $\s^{-1}$), then there cannot
occur proper cancellation in passing from $\psi(u)$ to $\s\psi(u)$
(or $\psi(u)$ to $\s^{-1}\psi(u)$). Also if $\psi$ contains at
least $\|u\|$ factors of $\t$ (or $\t^{-1}$), then there cannot
occur proper cancellation in passing from $\psi(u)$ to $\t\psi(u)$
(or $\psi(u)$ to $\t^{-1}\psi(u)$).
\end{lemma}

The main result of this section is

\begin{theorem} \label{pro:3.1} Let $u$ be an element in $F_2$, and let $\Omega$ be the
set of all chains of type (C1) or (C2) of length less than or
equal to $2\|u\|+3$. Suppose that the cyclic word $[\phi(u)]$ is
positive for some automorphism $\phi$ of $F_2$. Then there exists
$\psi \in \Omega$ and a Whitehead automorphism $\b$ of $F_2$ of
type (W1) such that the cyclic word $[\b \psi(u)]$ is positive
(which is obviously equivalent to saying that there exists $c \in
F_2$ such that $\pi_c \b \psi(u)$ is positive, where $\pi_c$ is
the inner automorphism of $F_2$ induced by $c$).
\end{theorem}

Once this theorem is proved, an algorithm to decide whether or not
a given word in $F_2$ is potentially positive is naturally derived
as follows.

\begin{algorithm} Let $u$ be an element in $F_2$, and let $\Omega$ be
defined as in the statement of Theorem ~\ref{pro:3.1}. Clearly
$\Omega$ is a finite set. Check if there is $\psi \in \Omega$ and
a Whitehead automorphism $\b$ of $F_2$ of type (W1) for which the
cyclic word $[\b \psi(u)]$ is positive. If so, conclude that $u$
is potentially positive; otherwise conclude that $u$ is not
potentially positive.
\end{algorithm}

\medskip
\noindent {\it Proof of Theorem ~\ref{pro:3.1}.} By Lemma
~\ref{lem:2.3}, $\phi$ can be expressed as
$$\phi \equiv \b
\phi',$$ where $\b$ is a Whitehead automorphism of $F_2$ of type
(W1) and $\phi'$ is a chain of type (C1) or (C2). By the
hypothesis of the theorem,

\begin{equation} \label{equ:3.1}
\text{\rm $[\phi(u)]=[\b \phi' (u)]$ is positive.}
\end{equation}

If $|\phi'| \le 2\|u\|+3$, then there is nothing to prove. So
suppose that $|\phi'|> 2\|u\|+3$. We proceed with the proof by
induction on $|\phi'|$. Assume that $\phi'$ is a chain of type
(C1) which ends in $\t$ (the other cases are analogous). Write
$$\phi'=\t \phi_1,$$
where $\phi_1$ is a chain of type (C1). Since $|\phi_1| \ge
2\|u\|+3$, $\phi_1$ must contain at least $\|u\|+2$ factors of
$\s$ or $\t$. We consider two cases separately.

\medskip
\noindent {\bf Case 1.} {\it $\s$ occurs at least $\|u\|+2$ times
in $\phi_1$.}
\medskip

Write
$$\phi_1=\t^{m_t} \s^{\ell_t} \cdots \t^{m_1} \s^{\ell_1},$$
where all $m_i, \ell_i > 0$ but $\ell_1$ and $m_t$ may be zero.

\medskip
\noindent {\bf Case 1.1.} {\it $m_t \ge 1$.}
\medskip

In this case, put
$$\phi_1=\t^{m_t} \phi_2,$$
where $\phi_2$ is a chain of type (C1). By Lemma ~\ref{lem:2.4},
no proper cancellation can occur in passing from $[\s^{\ell_t-1}
\cdots \t^{m_1} \s^{\ell_1}(u)]$ to $[\phi_2(u)]$, and hence the
cyclic word $[\phi_2(u)]$ does not contain a subword of the form
$a^2$ or $a^{-2}$. From this fact and the assumption $m_t \ge 1$,
we can observe that no proper cancellation occurs in passing from
$[\phi_1(u)]$ to $[\t \phi_1(u)]=[\phi'(u)]$. This implies from
(\ref{equ:3.1}) that the cyclic word $[\b \phi_1(u)]$ is positive,
and thus induction completes the case.

\medskip
\noindent {\bf Case 1.2.} {\it $m_t=0$.}
\medskip

In this case, we may put
$$\phi_1=\s \phi_3,$$
where $\phi_3$ is a chain of type (C1). Again by Lemma
~\ref{lem:2.4}, no proper cancellation can occur in passing from
$[\phi_3(u)]$ to $[\s \phi_3(u)]=[\phi_1(u)]$. Additionally, the
proof of Theorem ~1.2 of ~\cite{lee} shows that proper
cancellation occurs in passing from $[\phi_3(u)]$ to $[\t
\phi_3(u)]$ exactly in the same place where proper cancellation
occurs in passing from $[\phi_1(u)]$ to $[\t
\phi_1(u)]=[\phi'(u)]$. Therefore, by (\ref{equ:3.1}), the cyclic
word $[\b \t \phi_3 (u)]$ is positive. Since $|\t
\phi_3|=|\phi'|-1$, we are done by induction.

\medskip
\noindent {\bf Case 2.} {\it $\t$ occurs at least $\|u\|+2$ times
in $\phi_1$.}
\medskip

In this case, also by Lemma ~\ref{lem:2.4}, no proper cancellation
can occur in passing from $[\phi_1(u)]$ to $[\t
\phi_1(u)]=[\phi'(u)]$. It then follows from (\ref{equ:3.1}) that
the cyclic word $[\b \phi_1(u)]$ is positive; hence the required
result follows by induction. \qed

\section{Bounded translation equivalence in $F_2$}

We begin this section by fixing notation. Following
~\cite{kapovich}, if $w$ is a cyclic word in $F_2$ and $x, y \in
\{a, b\}^{\pm 1}$, we use $n(w; x, y)$ to denote the total number
of occurrences of the subwords $xy$ and $y^{-1}x^{-1}$ in $w$.
Then clearly $n(w; x, y)=n(w; y^{-1}, x^{-1})$. Similarly we denote
by $n(w;x)$ the total number of occurrences of $x$ and $x^{-1}$ in
$w$. Again clearly $n(w;x)=n(w; x^{-1})$.

In this section, we shall prove that there exists an algorithm to
determine bounded translation equivalence in $F_2$. Let $u \in
F_2$. We first establish four preliminary lemmas which demonstrate
the difference between $\|\s \psi(u)\|$ or $\|\t \psi(u)\|$ and
$\|\psi(u)\|$, and which describe the situation when this
difference becomes zero, in the case where $\psi$ is a chain of
type (C1) that contains a number of factors of $\s$. We remark
that similar statements to the lemmas also hold if $\s$ and $\t$
are interchanged with each other, or (C1) is replaced by (C2) and
$\s$ and $\t$ are replaced by $\s^{-1}$ and $\t^{-1}$,
respectively.

\begin{lemma} \label{lem:4.1}
Let $u \in F_2$. Suppose that $\psi$ is a chain of type (C1)
which contains at least $\|u\|+2$ factors of $\s$. We may write
$\psi=\t^m \s \psi_1$, where $m \ge 0$ and $\psi_1$ is a chain of
type (C1).
Then
\[
\begin{aligned}
(i) \ \|\s \psi(u)\|-\|\psi(u)\|&=\|\s \t^m \psi_1(u)\|-\|\t^m
\psi_1(u)\|+m(\|\s \psi_1(u)\|-\|\psi_1(u)\|); \\
(ii) \|\t \psi(u)\|-\|\psi(u)\|&=\|\t \t^m \psi_1(u)\|-\|\t^m
\psi_1(u)\|+\|\s \psi_1(u)\|-\|\psi_1(u)\|.
\end{aligned}
\]
\end{lemma}

\begin{proof}
By the proof of Case ~1 of Theorem ~1.2 in ~\cite{lee}, we see
that

\begin{equation} \label{equ:4.1}
n([\t^i \s \psi_1(u)]; b, a^{-1})=n([\t^i \psi_1(u)]; b, a^{-1})
\end{equation}
for every $i \ge 0$, because $\psi_1$ contains at least $\|u\|+1$
factors of $\s$. In particular,
\begin{equation} \label{equ:4.1.5}
n([\psi(u)]; b, a^{-1})=n([\t^m \psi_1(u)]; b, a^{-1}),
\end{equation}
for $\psi=\t^m \s \psi_1$. Since only $a$ or $a^{-1}$ can possibly
cancel or newly occur in the process of applying $\t$, the number
of $b$ and $b^{-1}$ remains unchanged if $\t$ is applied. Thus
\begin{equation} \label{equ:4.2}
\begin{aligned}
n([\t^i \s \psi_1(u)]; b)&=n([\s \psi_1(u)]; b); \\
n([\t^i \psi_1(u)]; b)&=n([\psi_1(u)]; b)
\end{aligned}
\end{equation}
for every $i \ge 0$. Also since only $b$ or $b^{-1}$ can possibly
cancel or newly occur in the process of applying $\s$, we get
\[
n([\s \psi_1(u)]; b)=n([\psi_1(u)]; b)+\|\s
\psi_1(u)\|-\|\psi_1(u)\|.
\]
By (\ref{equ:4.2}), this equality can be rewritten as
\begin{equation} \label{equ:4.3}
n([\t^i \s \psi_1(u)]; b)=n([\t^i \psi_1(u)]; b)+\|\s
\psi_1(u)\|-\|\psi_1(u)\|
\end{equation}
for every $i \ge 0$. In particular,
\begin{equation} \label{equ:4.3.5}
n([\psi(u)]; b)=n([\t^m \psi_1(u)]; b)+\|\s
\psi_1(u)\|-\|\psi_1(u)\|,
\end{equation}
for $\psi=\t^m \s \psi_1$.

Equality (\ref{equ:4.3}) together with (\ref{equ:4.1}) yields that
\begin{multline} \label{equ:4.4}
n([\t^i \s \psi_1(u)]; b)-n([\t^i \s \psi_1(u)]; b, a^{-1}) \\
= n([\t^i \psi_1(u)]; b)-n([\t^i \psi_1(u)]; b, a^{-1})+\|\s
\psi_1(u)\|-\|\psi_1(u)\|
\end{multline}
for every $i \ge 0$. Here, since
\[
\begin{aligned}
\|\t^{i+1} \s \psi_1(u)\|-\|\t^i \s \psi_1(u)\|&=n([\t^i \s
\psi_1(u)];b)-2n([\t^i \s \psi_1(u)]; b, a^{-1}); \\
\|\t^{i+1} \psi_1(u)\|-\|\t^i \psi_1(u)\|&=n([\t^i \psi_1(u)];
b)-2n([\t^i \psi_1(u)]; b, a^{-1}),
\end{aligned}
\]
equality (\ref{equ:4.4}) can be rephrased as
\[
\|\t^{i+1} \s \psi_1(u)\|-\|\t^i \s \psi_1(u)\|=\|\t^{i+1}
\psi_1(u)\|-\|\t^i \psi_1(u)\|+\|\s \psi_1(u)\|-\|\psi_1(u)\|
\]
for every $i \ge 0$. By summing up both sides of these equalities
changing $i$ from $0$ to $m-1$, we have
\[
\|\t^m \s \psi_1(u)\|-\|\s \psi_1(u)\|=\|\t^m
\psi_1(u)\|-\|\psi_1(u)\|+m(\|\s \psi_1(u)\|-\|\psi_1(u)\|),
\]
so that
\begin{equation} \label{equ:4.7}
\|\t^m \s \psi_1(u)\|-\|\t^m \psi_1(u)\|=(m+1)(\|\s
\psi_1(u)\|-\|\psi_1(u)\|).
\end{equation}
Since $\psi=\t^m \s \psi_1$, equality (\ref{equ:4.7}) can be
rephrased as
\begin{equation} \label{equ:4.8}
\|\psi(u)\|-\|\t^m \psi_1(u)\|=(m+1)(\|\s
\psi_1(u)\|-\|\psi_1(u)\|).
\end{equation}

Clearly
\[
\begin{aligned}
n([\psi(u)]; a)&=\|\psi(u)\|-n([\psi(u)]; b); \\
n([\t^m \psi_1(u)]; a)&=\|\t^m \psi_1(u)\|-n([\t^m \psi_1(u)]; b).
\end{aligned}
\]
These equalities together with (\ref{equ:4.3.5}) and
(\ref{equ:4.8}) yield that
\begin{equation} \label{equ:4.10}
n([\psi(u)]; a)=n([\t^m \psi_1(u)]; a)+m(\|\s
\psi_1(u)\|-\|\psi_1(u)\|).
\end{equation}
It then follows from
\[
\begin{aligned}
\|\s \psi(u)\|-\|\psi(u)\|&=n([
\psi(u)]; a)-2n([\psi(u)]; a, b^{-1}); \\
\|\s \t^m \psi_1(u)\|-\|\t^m \psi_1(u)\|&=n([\t^m \psi_1(u)];
a)-2n([\t^m \psi_1(u)]; a, b^{-1})
\end{aligned}
\]
together with (\ref{equ:4.1.5}) and (\ref{equ:4.10}) that
\[
\|\s \psi(u)\|-\|\psi(u)\|=\|\s \t^m \psi_1(u)\|-\|\t^m
\psi_1(u)\|+m(\|\s \psi_1(u)\|-\|\psi_1(u)\|),
\]
thus proving the first assertion of the lemma.

On the other hand, we deduce from
\[
\begin{aligned}
\|\t \psi(u)\|-\|\psi(u)\|&=n([
\psi(u)]; b)-2n([\psi(u)]; b, a^{-1}); \\
\|\t \t^m \psi_1(u)\|-\|\t^m \psi_1(u)\|&=n([\t^m \psi_1(u)];
b)-2n([\t^m \psi_1(u)]; b, a^{-1})
\end{aligned}
\]
together with (\ref{equ:4.1.5}) and (\ref{equ:4.3.5}) that
\[
\|\t \psi(u)\|-\|\psi(u)\|=\|\t \t^m \psi_1(u)\|-\|\t^m
\psi_1(u)\|+\|\s \psi_1(u)\|-\|\psi_1(u)\|,
\]
which proves the
second assertion of the lemma.
\end{proof}

\begin{lemma}  \label{lem:4.3} Let $u \in F_2$. Suppose that $\psi$ is a chain
of type (C1) which contains at least $\|u\|$ factors of $\s$. Then
\[
\begin{aligned}
&(i) \ \|\s \psi(u)\|-\|\psi(u)\| \ge 0; \\
&(ii) \ \|\t \psi(u)\|-\|\psi(u)\| \ge 0.
\end{aligned}
\]
\end{lemma}

\begin{proof} Clearly
\begin{equation} \label{equ:4.11}
\begin{split}
\|\s \psi(u)\|-\|\psi(u)\|&=n([\psi(u)];a)-2n([\psi(u); a,b^{-1})
\\
&=n([\psi(u)];a, a)+n([\psi(u)];a, b)-n([\psi(u)];a, b^{-1}).
\end{split}
\end{equation}
Since $\psi$ contains at least $\|u\|$ factors of $\s$, by Lemma
~\ref{lem:2.4}, there cannot occur proper cancellation in passing
from $[\psi(u)]$ to $[\s\psi(u)]$. Hence every subword of
$[\psi(u)]$ of the form $ab^{-1}$ or $ba^{-1}$ is necessarily part
of a subword of the form $ab^{-r}a^{-1}$ or $ab^ra^{-1}$ ($r >0$),
respectively. This implies that
\[
n([\psi(u)];a, b) \ge n([\psi(u)];a, b^{-1}),
\]
so that, from (\ref{equ:4.11}),
\[
\|\s \psi(u)\|-\|\psi(u)\| \ge n([\psi(u)];a, a) \ge 0,
\]
thus proving (i).

On the other hand, clearly
\begin{equation} \label{equ:4.12}
\begin{split}
\|\t \psi(u)\|-\|\psi(u)\|&=n([\psi(u)];b)-2n([\psi(u); b,a^{-1})
\\
&=n([\psi(u)];b, b)+n([\psi(u)];b, a)-n([\psi(u)];b, a^{-1}).
\end{split}
\end{equation}
As above, every subword of $[\psi(u)]$ of the form $ab^{-1}$ or
$ba^{-1}$ is necessarily part of a subword of the form
$ab^{-r}a^{-1}$ or $ab^ra^{-1}$ ($r >0$), respectively. Observe
that a subword of $[\psi(u)]$ of the form $ab^{\pm r}a^{-1}$ is
actually part of either a subword of the form $ba^sb^{\pm
r}a^{-1}$ or a subword of the form $a^{-1}b^{-t}a^sb^{\pm
r}a^{-1}$ ($s, t >0$). This implies that
\begin{equation} \label{equ:4.12.5}
n([\psi(u)];b, a) \ge n([\psi(u)];b, a^{-1}),
\end{equation}
so that, from (\ref{equ:4.12}),
\begin{equation} \label{equ:4.13}
\|\t \psi(u)\|-\|\psi(u)\| \ge n([\psi(u)];b, b) \ge 0,
\end{equation}
thus proving (ii).
\end{proof}

\begin{lemma}  \label{lem:4.3.3} Let $u \in F_2$. Suppose that $\psi$ is a chain
of type (C1) which contains at least $\|u\|+1$ factors of $\s$.
Then

(i) if $\|\s \psi(u)\|=\|\psi(u)\|$, then $\|\s^{i+1}
\psi(u)\|=\|\s^i \psi(u)\|$ for every $i \ge 0$;

(ii) if $\|\s^{j+1} \psi(u)\|=\|\s^j \psi(u)\|$ for some $j \ge
0$, then $\|\s \psi(u)\|=\|\psi(u)\|$.
\end{lemma}

\begin{proof}
For (i), assume that $\|\s \psi(u)\|=\|\psi(u)\|$. We shall prove
$\|\s^{i+1} \psi(u)\|=\|\s^i \psi(u)\|$ by induction on $i \ge 0$.
The case where $i=0$ is clear. So let $i \ge 1$. By Lemma
~\ref{lem:4.1} ~(i) with $m=0$, we have
\[
\|\s^{i+1} \psi(u)\|-\|\s^i \psi(u)\|=\|\s^i \psi(u)\|-\|\s^{i-1}
\psi(u)\|.
\]
It follows from the induction hypothesis that
\[\|\s^{i+1} \psi(u)\|=\|\s^i \psi(u)\|,
\]
so proving (i).

For (ii), assume that $\|\s^{j+1} \psi(u)\|=\|\s^j \psi(u)\|$ for
some $j \ge 0$. We use induction on $j \ge 0$. If $j=0$, then
there is nothing to prove. So let $j \ge 1$. It follows from Lemma
~\ref{lem:4.1} ~(i) with $m=0$ that
\[
0=\|\s^{j+1} \psi(u)\|-\|\s^j \psi(u)\|=\|\s^j
\psi(u)\|-\|\s^{j-1} \psi(u)\|,
\]
so that
\[\|\s^j \psi(u)\|=\|\s^{j-1} \psi(u)\|.\]
Then by the induction hypothesis, we get the required result.
\end{proof}

\begin{lemma}  \label{lem:4.3.5} Let $u \in F_2$, and let $\psi=\s \psi_1$, where
$\psi_1$ is a chain of type (C1) which contains at least $\|u\|+1$
factors of $\s$. Suppose that $\|\t \psi(u)\|=\|\psi(u)\|$. Then
$\|\s^{i+1} \psi_1(u)\|=\|\s^i \psi_1(u)\|$ for every $i \ge 0$.
\end{lemma}

\begin{proof}
By Lemma ~\ref{lem:4.1} ~(ii) with $m=0$, we have
\[
0=\|\t \psi(u)\|-\|\psi(u)\|=\|\t \psi_1(u)\|-\|\psi_1(u)\|+\|\s
\psi_1(u)\|-\|\psi_1(u)\|.
\]
Here, by Lemma ~\ref{lem:4.3} ~(ii), $\|\t
\psi_1(u)\|-\|\psi_1(u)\| \ge 0$. Also by Lemma ~\ref{lem:4.3}
~(i), $\|\s \psi_1(u)\|-\|\psi_1(u)\| \ge 0$. Hence we must have
\[\|\t \psi_1(u)\|=\|\psi_1(u)\| \ \text{\rm and} \ \|\s
\psi_1(u)\|=\|\psi_1(u)\|.\] The second equality $\|\s
\psi_1(u)\|=\|\psi_1(u)\|$ yields from Lemma ~\ref{lem:4.3.3} ~(i)
that
\[ \|\s^{i+1} \psi_1(u)\|=\|\s^i \psi_1(u)\| \]
for every $i \ge 0$, thus proving the assertion.
\end{proof}

For the proof of the main result of the present section, we need
the following two technical corollaries of Lemmas
~\ref{lem:4.1}--\ref{lem:4.3.5}. We remark that similar statements
to the corollaries also hold if $\s$ and $\t$ are interchanged
with each other, or (C1) is replaced by (C2) and $\s$ and $\t$ are
replaced by $\s^{-1}$ and $\t^{-1}$, respectively.

\begin{corollary} \label{cor:4.4.5} Let $u, v \in F_2$ with $\|u\| \ge \|v\|$,
and let $\psi$ be a chain of type (C1) with $|\psi| \ge 2\|u\|+3$.
Put $k=\|u\|+1$. Suppose that $u$ and $v$ have the property that
\[
\begin{aligned}
&\|\s^{k+1} \psi'(u)\|=\|\s^k \psi'(u)\| \ \text{if and only
if} \ \|\s^{k+1} \psi'(v)\|=\|\s^k \psi'(v)\|; \\
&\|\t^{k+1} \psi'(u)\|=\|\t^k \psi'(u)\| \ \text{if and only if} \
\|\t^{k+1} \psi'(v)\|=\|\t^k \psi'(v)\|,
\end{aligned}
\]
for every chain $\psi'$ of type (C1) with $|\psi'| < |\psi|$. Then
we have
\[
\begin{aligned}
&(i) \ \|\s^{k+1} \psi(u)\|=\|\s^k \psi(u)\| \ \text{if and only
if} \ \|\s^{k+1} \psi(v)\|=\|\s^k \psi(v)\|; \\
&(ii) \ \|\t^{k+1} \psi(u)\|=\|\t^k \psi(u)\| \ \text{if and only
if} \ \|\t^{k+1} \psi(v)\|=\|\t^k \psi(v)\|.
\end{aligned}
\]
\end{corollary}

\begin{proof} Suppose that $\psi$ ends in $\t$ (the case where $\psi$ ends in $\s$
is analogous). Since $|\psi| \ge 2\|u\|+3$, either $\s$ or $\t$
occurs at least $\|u\|+2$ times in $\psi$. We consider two cases
separately.

\medskip
{\bf Case 1.} $\s$ occurs at least $\|u\|+2$ times in $\psi$.
\medskip

First we shall prove (i). Suppose that $\|\s^{k+1}
\psi(u)\|=\|\s^k \psi(u)\|$. By Lemma ~\ref{lem:4.3.3} ~(ii), we
have
\[
\|\s \psi(u)\|=\|\psi(u)\|.
\]
Write
\[\psi=\t^{\ell} \s \psi_1,\]
where $\ell \ge 1$ and $\psi_1$ is a chain of type (C1). Clearly
$\psi_1$ contains at least $\|u\|+1$ factors of $\s$. By Lemma
~\ref{lem:4.1} ~(i), we have
\[
0=\|\s \psi(u)\|-\|\psi(u)\|=\|\s \t^{\ell}
\psi_1(u)\|-\|\t^{\ell} \psi_1(u)\|+{\ell}(\|\s
\psi_1(u)\|-\|\psi_1(u)\|).\] Here, since $\|\s \t^{\ell}
\psi_1(u)\|-\|\t^{\ell} \psi_1(u)\| \ge 0$ and $\|\s
\psi_1(u)\|-\|\psi_1(u)\| \ge 0$ by Lemma ~\ref{lem:4.3} ~(i), the
only possibility is that
\[\|\s \t^{\ell} \psi_1(u)\|=\|\t^{\ell} \psi_1(u)\| \ \text{\rm and} \ \|\s
\psi_1(u)\|=\|\psi_1(u)\|.
\]
These equalities together with Lemma ~\ref{lem:4.3.3} ~(i) yield
that
\[
\|\s^{k+1} \t^{\ell} \psi_1(u)\|=\|\s^k \t^{\ell} \psi_1(u)\| \
\text{\rm and} \ \|\s^{k+1} \psi_1(u)\|=\|\s^k \psi_1(u)\|. \]
Since $|\t^{\ell} \psi_1|<|\psi|$ and $|\psi_1|<|\psi|$, by the
hypothesis of the corollary, we get
\[
\|\s^{k+1} \t^{\ell} \psi_1(v)\|=\|\s^k \t^{\ell} \psi_1(v)\| \
\text{and} \  \|\s^{k+1} \psi_1(v)\|=\|\s^k \psi_1(v)\|.
\]
Again by Lemma ~\ref{lem:4.3.3} ~(ii), we have
\[
\|\s \t^{\ell} \psi_1(v)\|=\|\t^{\ell} \psi_1(v)\| \ \text{and} \
\|\s \psi_1(v)\|=\|\psi_1(v)\|.
\]
Therefore, by Lemma ~\ref{lem:4.1} ~(i),
\[ \begin{aligned}
\|\s \psi(v)\|-\|\psi(v)\|&=\|\s \t^{\ell} \psi_1(v)\|-\|\t^{\ell}
\psi_1(v)\|+{\ell}(\|\s \psi_1(v)\|-\|\psi_1(v)\|)\\
&=0,
\end{aligned}
\]
namely, $\|\s \psi(v)\|=\|\psi(v)\|$. Then the desired equality
$\|\s^{k+1} \psi(v)\|=\|\s^k \psi(v)\|$ follows from Lemma
~\ref{lem:4.3.3} ~(i).

Conversely, if $\|\s^{k+1} \psi(v)\|=\|\s^k \psi(v)\|$, we can
deduce, in the same way as above, that $\|\s^{k+1}
\psi(u)\|=\|\s^k \psi(u)\|$.

Next we shall prove (ii). Assume that $\|\t^{k+1} \psi(u)\|=\|\t^k
\psi(u)\|$. Apply Lemma ~\ref{lem:4.1} ~(ii) to get
\begin{equation} \label{equ:4.13.3}
0=\|\t^{k+1} \psi(u)\|-\|\t^k \psi(u)\|=\|\t^{k+1} \t^{\ell}
\psi_1(u)\|-\|\t^k \t^{\ell} \psi_1(u)\|+\|\s
\psi_1(u)\|-\|\psi_1(u)\|.
\end{equation}
Here, since $\|\t^{k+1} \t^{\ell} \psi_1(u)\|-\|\t^k \t^{\ell}
\psi_1(u)\| \ge 0$ by Lemma ~\ref{lem:4.3} ~(ii), and since $\|\s
\psi_1(u)\|-\|\psi_1(u)\| \ge 0$ by Lemma ~\ref{lem:4.3} ~(i), we
must have
\begin{equation} \label{equ:4.13.5}
\|\t^{k+1} \t^{\ell} \psi_1(u)\|=\|\t^k \t^{\ell} \psi_1(u)\| \
\text{\rm and} \ \|\s \psi_1(u)\|=\|\psi_1(u)\|.
\end{equation}
Since $|\t^{\ell} \psi_1|<|\psi|$, by the hypothesis of the
corollary, the first equality of (\ref{equ:4.13.5}) implies that
\[\|\t^{k+1} \t^{\ell} \psi_1(v)\|=\|\t^k \t^{\ell} \psi_1(v)\|.\]
Also, from the second equality of (\ref{equ:4.13.5}), arguing as
above, we deduce that
\[
\|\s \psi_1(v)\|=\|\psi_1(v)\|.
\]
Therefore, by Lemma ~\ref{lem:4.1} ~(ii),
\[ \begin{aligned}
\|\t^{k+1} \psi(v)\|-\|\t^k \psi(v)\|&=\|\t^{k+1} \t^{\ell}
\psi_1(v)\|-\|\t^k \t^{\ell} \psi_1(v)\|+\|\s
\psi_1(v)\|-\|\psi_1(v)\| \\
&=0,
\end{aligned}
\]
that is, $\|\t^{k+1} \psi(v)\|=\|\t^k \psi(v)\|$, as required.

It is clear that the converse is also true.

\medskip
{\bf Case 2.} $\t$ occurs at least $\|u\|+2$ times in $\psi$.
\medskip

Since $\psi$ is assumed to end in $\t$, we may write
$$\psi=\t \psi_2,$$
where $\psi_2$ is a chain of type (C1) that contains at least
$\|u\|+1$ factors of $\t$.

First we shall prove (i). Suppose that $\|\s^{k+1}
\psi(u)\|=\|\s^k \psi(u)\|$. By Lemma ~\ref{lem:4.1} ~(ii) with
$\s, \t$ interchanged, we have
\[
0=\|\s^{k+1} \psi(u)\|-\|\s^k \psi(u)\|=\|\s^{k+1}
\psi_2(u)\|-\|\s^k \psi_2(u)\|+\|\t \psi_2(u)\|-\|\psi_2(u)\|.
\]
This is a similar situation to (\ref{equ:4.13.3}) with $\s, \t$
interchanged. So arguing as in Case ~1, we get the desired
equality $\|\s^{k+1} \psi(v)\|=\|\s^k \psi(v)\|$. Clearly the
converse also holds.

Next we shall prove (ii). Suppose that $\|\t^{k+1}
\psi(u)\|=\|\t^k \psi(u)\|$. By Lemma ~\ref{lem:4.1} ~(i) with
$\s, \t$ interchanged and $m=0$, we have
\[
0=\|\t^{k+1} \psi(u)\|-\|\t^k \psi(u)\|=\|\t^k \psi(u)\|-\|\t^{k-1}
\psi(u)\|.
\]
So
\[\|\t^k \psi(u)\|=\|\t^{k-1}
\psi(u)\|.\] This equality can be rephrased as
\[\|\t^{k+1} \psi_2(u)\|=\|\t^k
\psi_2(u)\|,\] because $\psi=\t \psi_2$. Since $|\psi_2|<|\psi|$,
by the hypothesis of the corollary,
\[\|\t^{k+1} \psi_2(v)\|=\|\t^k
\psi_2(v)\|,\] that is,
\[\|\t^k \psi(v)\|=\|\t^{k-1}
\psi(v)\|.\] Thus, by Lemma ~\ref{lem:4.1} ~(i) with $\s, \t$
interchanged and $m=0$, we obtain
\[\|\t^{k+1} \psi(v)\|-\|\t^k
\psi(v)\|=\|\t^k \psi(v)\|-\|\t^{k-1} \psi(v)\|=0,
\]
namely, $\|\t^{k+1} \psi(v)\|=\|\t^k \psi(v)\|$, as required.
Obviously the converse is also true.
\end{proof}

\begin{corollary} \label{cor:4.5} Let $u, v \in F_2$ with $\|u\| \ge
\|v\|$, and let $\psi$ be a chain of type (C1). Put $k=\|u\|+1$.
Suppose that $u$ and $v$ have the property that
\[
\begin{aligned}
\|\s^{k+1} \psi'(u)\|=\|\s^k \psi'(u)\| \ \text{if and only if} \
\|\s^{k+1} \psi'(v)\|=\|\s^k \psi'(v)\|; \\
\|\t^{k+1} \psi'(u)\|=\|\t^k \psi'(u)\| \ \text{if and only if} \
\|\t^{k+1} \psi'(v)\|=\|\t^k \psi'(v)\|,
\end{aligned}
\]
for every chain $\psi'$ of type (C1) with $|\psi'| \le |\psi|$.
Then we have

(i) if $\psi$ contains at least $\|u\|+1$ factors of $\s$, then
\[
\|\s \psi(u)\|=\|\psi(u)\| \ \text{if and only if} \ \|\s
\psi(v)\|=\|\psi(v)\|;
\]

(ii) if $\|\t \psi(u)\|=\|\psi(u)\|$ or $\|\t
\psi(v)\|=\|\psi(v)\|$, and $\psi=\s \psi_1$, where $\psi_1$ is a
chain of type (C1) which contains at least $\|u\|+1$ factors of
$\s$, then
\[
\|\s \psi_1(u)\|=\|\psi_1(u)\| \ \text{and} \ \|\s
\psi_1(v)\|=\|\psi_1(v)\|;
\]

(iii) if $\psi$ contains at least $\|u\|+2$ factors of $\s$ and
ends in $\t$, then
\[
\|\t \psi(u)\|=\|\psi(u)\| \ \text{if and only if} \ \|\t
\psi(v)\|=\|\psi(v)\|.
\]
\end{corollary}

\begin{proof}
For (i), let $\psi$ contain at least $\|u\|+1$ factors of $\s$,
and suppose that $\|\s \psi(u)\|=\|\psi(u)\|$. By Lemma
~\ref{lem:4.3.3} ~(i), we have $\|\s^{k+1} \psi(u)\|=\|\s^k
\psi(u)\|$. Then by the hypothesis of the corollary, $\|\s^{k+1}
\psi(v)\|=\|\s^k \psi(v)\|$. Finally by Lemma ~\ref{lem:4.3.3}
~(ii), we get $\|\s \psi(v)\|=\|\psi(v)\|$. The converse also
holds.

For (ii), let $\psi=\s \psi_1$, where $\psi_1$ is a chain of type
(C1) containing at least $\|u\|+1$ factors of $\s$, and suppose
that $\|\t \psi(u)\|=\|\psi(u)\|$. By Lemma ~\ref{lem:4.3.5}, we
have $\|\s \psi_1(u)\|=\|\psi_1(u)\|$. Then, by (i) of the
corollary, $\|\s \psi_1(v)\|=\|\psi_1(v)\|$. The converse is
proved similarly.

For (iii), let $\psi$ contain at least $\|u\|+2$ factors of $\s$,
and let $\psi$ end in $\t$. Assume that $\|\t
\psi(u)\|=\|\psi(u)\|$. Write
\[
\psi=\t^{\ell} \s \psi_2,
\]
where $\ell \ge 1$ and $\psi_2$ is a chain of type (C1). By Lemma
~\ref{lem:4.1} ~(ii), we have
\[
0=\|\t \psi(u)\|-\|\psi(u)\|=\|\t \t^{\ell} \psi_2
(u)\|-\|\t^{\ell} \psi_2(u)\|+\|\s \psi_2(u)\|-\|\psi_2(u)\|.
\]
Here, since $\|\t \t^{\ell} \psi_2 (u)\|-\|\t^{\ell} \psi_2(u)\|
\ge 0$ by Lemma ~\ref{lem:4.3} ~(ii) and $\|\s
\psi_2(u)\|-\|\psi_2(u)\| \ge 0$ by Lemma ~\ref{lem:4.3} ~(i), we
must have
\[
\|\t \t^{\ell} \psi_2 (u)\|=\|\t^{\ell} \psi_2(u)\| \ \text{\rm
and} \ \|\s \psi_2(u)\|=\|\psi_2(u)\|.
\]
Since $\|\s \psi_2(u)\|=\|\psi_2(u)\|$, by (i) of the corollary,
\[
\|\s \psi_2(v)\|=\|\psi_2(v)\|.
\]
Also, the following claim shows that $\|\t \t^{\ell} \psi_2
(v)\|=\|\t^{\ell} \psi_2(v)\|$. Then by Lemma ~\ref{lem:4.1}
~(ii), we have $\|\t \psi(v)\|=\|\psi(v)\|$, as required.

\medskip
{\bf Claim.} $\|\t \t^{\ell} \psi_2 (v)\|=\|\t^{\ell}
\psi_2(v)\|$.
\medskip

{\it Proof of the Claim.} Since $\|\t \psi(u)\|=\|\psi(u)\|$, in
view of (\ref{equ:4.12}), (\ref{equ:4.12.5}) and (\ref{equ:4.13})
in the proof of Lemma ~\ref{lem:4.3}, we must have
\begin{equation} \label{equ:4.14}
\text {\rm $n([\psi(u)];b, a)=n([\psi(u)];b, a^{-1})$ and
$n([\psi(u)];b, b)=0$.}
\end{equation}
Since the chain $\psi_2$ contains at least $\|u\|+1$ factors of
$\s$, by Lemma ~\ref{lem:2.4}, no proper cancellation occurs in
passing from $[\psi_2(u)]$ to $[\s \psi_2(u)]$. This yields that
\begin{equation} \label{equ:4.14.3}
\text{\rm $a^2$ or $a^{-2}$ cannot occur in $[\s \psi_2(u)]$ as a
subword.}
\end{equation}
From this, we see that, since $\ell \ge 1$,
\begin{equation} \label{equ:4.14.5}
\text{\rm no proper cancellation can occur in passing from
$[\psi(u)]$ to $[\t \psi(u)]$.}
\end{equation}
In view of (\ref{equ:4.14}), (\ref{equ:4.14.3}) and
(\ref{equ:4.14.5}), the cyclic word $[\psi(u)]$ must have the form
$$
[\psi(u)]=[a^{\epsilon}ba^{-\epsilon}b^{-1} \cdots
a^{\epsilon}ba^{-\epsilon}b^{-1}],
$$
where either $\epsilon=1$ or $\epsilon=-1$. Then, by applying
$\s^{-1}\t^{-\ell}$ to $[\psi(u)]$, we deduce that
$$
[\psi_2(u)]=[\psi(u)]=[a^{\epsilon}ba^{-\epsilon}b^{-1} \cdots
a^{\epsilon}ba^{-\epsilon}b^{-1}].
$$
It then follows that
\[
[\t^i \psi_2(u)]=[\psi_2(u)] \] for every $i \ge 0$, so that
\begin{equation} \label{equ:4.14.5.5}
\|\t^{i+1} \psi_2(u)\|=\|\t^i \psi_2(u)\|
\end{equation}
for every $i \ge 0$. In particular,
\[
\|\t^{k+1} \psi_2(u)\|=\|\t^k \psi_2(u)\|.
\]

So by the hypothesis of the corollary,
\begin{equation} \label{equ:4.14.6}
\|\t^{k+1} \psi_2(v)\|=\|\t^k \psi_2(v)\|.
\end{equation}
Then in the same way as obtaining (\ref{equ:4.14}), we get
\begin{equation} \label{equ:4.14.8}
\text {\rm $n([\t^k \psi_2(v)];b, a)=n([\t^k \psi_2(v)];b,
a^{-1})$ and $n([\t^k \psi_2(v)];b, b)=0$.}
\end{equation}
Since the chain $\t^k \psi_2$ contains at least $\|v\|+1$ factors
of $\t$, by Lemma ~\ref{lem:2.4}, no proper cancellation may occur
in passing from $[\t^k \psi_2(v)]$ to $[\t^{k+1} \psi_2(v)]$. This
together with (\ref{equ:4.14.8}) yields that
$$
[\t^k \psi_2(v)]=[a^{s_1}ba^{t_1}b^{-1} \cdots
a^{s_r}ba^{t_r}b^{-1}],
$$
where every $s_j, t_j$ is a nonzero integer. Then, by applying
$\t^{-k}$ to $[\t^k \psi_2(v)]$, we deduce that
$$
[\psi_2(v)]=[a^{s_1}ba^{t_1}b^{-1} \cdots a^{s_r}ba^{t_r}b^{-1}].
$$
Thus it follows that
\[
[\t^i \psi_2(v)]=[\psi_2(v)] \] for every $i \ge 0$, so that
\[
\|\t^{i+1} \psi_2(v)\|=\|\t^i \psi_2(v)\|
\]
for every $i \ge 0$. In particular, $\|\t \t^{\ell} \psi_2
(v)\|=\|\t^{\ell} \psi_2(v)\|$, as required. \qed

The proof of the corollary is now completed.
\end{proof}

For a Whitehead automorphism $\b$ of $F_2$, a chain $\psi$ of
Whitehead automorphisms of $F_2$ and an element $w$ in $F_2$, we
let $\|\b: \psi: w\|$ denote the maximum of $1$ and $\|\b \psi
(w)\|-\|\psi(w)\|$, that is,
$$\|\b: \psi: w\|:=\max\{1,\|\b \psi
(w)\|-\|\psi(w)\|\}.$$

Now we are ready to establish the main result of the present
section as follows.

\begin{theorem} \label{thm:4.7} Let $u, v \in F_2$ with $\|u\| \ge \|v\|$, and
let $\Omega$ be the set of all chains of type (C1) or (C2) of
length less than or equal to $2\|u\|+5$. Let $\Omega_1$ be the
subset of $\Omega$ consisting of all chains of type (C1), and let
$\Omega_2$ be the subset of $\Omega$ consisting of all chains of
type (C2). Put $k=\|u\|+1$. Suppose that $u$ and $v$ have the
property that
\[
\begin{aligned}
&\|\s^{k+1} \psi_1(u)\|=\|\s^k \psi_1(u)\| \ \text{if and only if}
\ \|\s^{k+1} \psi_1(v)\|=\|\s^k \psi_1(v)\|; \\
&\|\t^{k+1} \psi_1(u)\|=\|\t^k \psi_1(u)\| \ \text{if and only if}
\ \|\t^{k+1} \psi_1(v)\|=\|\t^k \psi_1(v)\|,
\end{aligned}
\]
for every $\psi_1 \in \Omega_1$, and that
\[
\begin{aligned}
&\|\s^{-k-1} \psi_2(u)\|=\|\s^{-k} \psi_2(u)\| \ \text{if and only
if} \ \|\s^{-k-1} \psi_2(v)\|=\|\s^{-k} \psi_2(v)\|; \\
&\|\t^{-k-1} \psi_2(u)\|=\|\t^{-k} \psi_2(u)\| \ \text{if and only
if} \ \|\t^{-k-1} \psi_2(v)\|=\|\t^{-k} \psi_2(v)\|,
\end{aligned}
\]
for every $\psi_2 \in \Omega_2$. Then $u$ and $v$ are boundedly
translation equivalent in $F_2$.

More specifically,
\[
\min \Delta \le {\|\phi(u)\| \over \|\phi(v)\|} \le \max \Delta
\]
for every automorphism $\phi$ of $F_2$, where
\[ \Delta := \{ {\|
\psi(u)\| \over \| \psi(v)\|}, {\| \a: \psi_1: u\| \over \| \a :
\psi_1: v \|}, {\| \a^{-1}: \psi_2: u\| \over \| \a^{-1}: \psi_2:
v\|} \, | \, \psi \in \Omega, \psi_i \in \Omega_i, \a=\s \
\text{or} \ \t \}.
\]
(Obviously, $\Delta$ is a finite set consisting of positive real
numbers.)
\end{theorem}

\begin{proof} Let $\phi$ be an automorphism of $F_2$. By Lemma
~\ref{lem:2.3}, $\phi$ can be represented as $$\phi \equiv
\b\phi',$$ where $\b$ is a Whitehead automorphism of $F_2$ of type
(W1) and $\phi'$ is of type either (C1) or (C2). We proceed with
the proof of the theorem by induction on $|\phi'|$. Letting
$\phi'$ be a chain of type (C1) with $|\phi'|>2\|u\|+5$ (the case
for (C2) is similar), assume that
\[
\begin{aligned}
&\|\s^{k+1} \psi(u)\|=\|\s^k \psi(u)\| \ \text{if and only if}
\ \|\s^{k+1} \psi(v)\|=\|\s^k \psi(v)\|; \\
&\|\t^{k+1} \psi(u)\|=\|\t^k \psi(u)\| \ \text{if and only if} \
\|\t^{k+1} \psi(v)\|=\|\t^k \psi(v)\|,
\end{aligned}
\]
and that
$$\min \Delta \le {\|\psi(u)\| \over
\|\psi(v)\|}, {\|\s: \psi: u\| \over \|\s : \psi: v\|}, {\|\t:
\psi: u\| \over \|\t : \psi: v\|} \le \max \Delta,$$ for every
chain $\psi$ of type (C1) with $|\psi| < |\phi'|$.

By Corollary ~\ref{cor:4.4.5}, it is easy to get
\[
\begin{aligned}
&\|\s^{k+1} \phi'(u)\|=\|\s^k \phi'(u)\| \ \text{if and only if}
\ \|\s^{k+1} \phi'(v)\|=\|\s^k \phi'(v)\|; \\
&\|\t^{k+1} \phi'(u)\|=\|\t^k \phi'(u)\| \ \text{if and only if} \
\|\t^{k+1} \phi'(v)\|=\|\t^k \phi'(v)\|.
\end{aligned}
\]
In the following Claims A, B and C, we shall prove that
$$\min \Delta \le {\|\phi'(u)\| \over \|\phi'(v)\|}, {\|\s: \phi': u\| \over \|\s : \phi': v\|},
{\|\t: \phi': u\| \over \|\t : \phi': v\|} \le \max \Delta,$$
which is clearly equivalent to showing that
$$\min \Delta \le
{\|\phi(u)\| \over \|\phi(v)\|}, {\|\s: \phi: u\| \over \|\s :
\phi: v\|}, {\|\t: \phi: u\| \over \|\t : \phi: v\|} \le \max
\Delta.$$

Suppose that $\phi'$ ends in $\t$ (the case where $\phi'$ ends in
$\s$ is analogous).

\medskip
\noindent
{\bf Claim A.}
$$\min \Delta \le {\|\phi'(u)\| \over \|\phi'(v)\|} \le \max \Delta$$
\medskip

\noindent {\it Proof of Claim ~A.} Since $\phi'$ ends in $\t$, we
may write
$$\phi'=\t \phi_1,$$
where $\phi_1$ is a chain of type (C1). Then obviously
\begin{equation}
\begin{aligned} \label{equ:a.1}
\| \phi'(u)\|&=\|\t \phi_1(u)\|-\|\phi_1(u)\|+\|\phi_1(u)\|; \\
\| \phi'(v)\|&=\|\t \phi_1(v)\|-\|\phi_1(v)\|+\|\phi_1(v)\|.
\end{aligned}
\end{equation}
If both $\|\t \phi_1(u)\| \neq \|\phi_1(u)\|$ and $\|\t
\phi_1(v)\| \neq \|\phi_1(v)\|$, then equalities (\ref{equ:a.1})
can be rephrased as
\begin{equation}
\begin{aligned} \label{equ:a.2}
\|\phi'(u)\|&=\|\t : \phi_1 :u\|+\|\phi_1(u)\|;
\\
\|\phi'(v)\|&=\|\t : \phi_1 : v\|+\|\phi_1(v)\|.
\end{aligned}
\end{equation}
Since
$$\min \Delta \le
{\|\phi_1(u)\| \over \|\phi_1(v)\|}, {\|\t: \phi_1: u\| \over \|\t
: \phi_1: v\|} \le \max \Delta$$ by the induction hypothesis, we
obtain
$$\min \Delta \le
{\|\phi'(u)\| \over \|\phi'(v)\|}\le \max \Delta,$$
as required.

So assume that
\begin{equation} \label{equ:a.3}
\|\t \phi_1(u)\|=\|\phi_1(u)\| \ \text{\rm or} \ \|\t
\phi_1(v)\|=\|\phi_1(v)\|.
\end{equation}
Clearly the chain $\phi_1$ has length $|\phi_1| = |\phi'|-1 \ge
2\|u\|+5$. Hence either $\s$ or $\t$ occurs at least $\|u\|+3$
times in $\phi_1$. We consider two cases accordingly.

\medskip
\noindent {\bf Case A.1.} $\s$ occurs at least $\|u\|+3$ times in
$\phi_1$.
\medskip

Since $\phi_1$ is a chain of type (C1), $\phi_1$ ends in either
$\s$ or $\t$.

\medskip
\noindent {\bf Case A.1.1.} $\phi_1$ ends in $\s$.
\medskip

Write
$$\phi_1=\s \phi_2,$$
where $\phi_2$ is a chain of type (C1). In view of Corollary
~\ref{cor:4.5} ~(ii), our assumption (\ref{equ:a.3}) yields that
\begin{equation} \label{equ:a.4}
\|\s \phi_2(u)\|=\|\phi_2(u)\|\ \text {\rm and} \ \|\s
\phi_2(v)\|=\|\phi_2(v)\|.
\end{equation}
This together with Lemma ~\ref{lem:4.1} ~(ii)
implies that
\begin{equation}
\begin{aligned} \label{equ:a.5}
\|\t \phi_1(u)\|-\|\phi_1(u)\|&=\|\t \phi_2(u)\|-\|\phi_2(u)\|; \\
\|\t \phi_1(v)\|-\|\phi_1(v)\|&=\|\t \phi_2(v)\|-\|\phi_2(v)\|.
\end{aligned}
\end{equation}
Since $\phi_1=\s \phi_2$, we obtain from (\ref{equ:a.4}) that
$\|\phi_1(u)\|=\|\phi_2(u)\|$ and $\|\phi_1(v)\|=\|\phi_2(v)\|$,
so that, from (\ref{equ:a.5}),
\begin{equation}
\begin{aligned} \label{equ:a.6}
\|\t \phi_1(u)\|&=\|\t \phi_2(u)\|; \\
\|\t \phi_1(v)\|&=\|\t \phi_2(v)\|.
\end{aligned}
\end{equation}
Since $\phi'=\t \phi_1$, (\ref{equ:a.6}) implies that
$${\|\phi'(u)\| \over \|\phi'(v)\|}={\|\t \phi_2(u)\| \over \|\t \phi_2(v)\|},$$
and thus, by the induction hypothesis,
$$\min \Delta \le {\|\phi'(u)\| \over \|\phi'(v)\|} \le \max \Delta,$$
as desired.

\medskip
\noindent {\bf Case A.1.2.} $\phi_1$ ends in $\t$.
\medskip

In view of Corollary ~\ref{cor:4.5} ~(iii), our assumption
(\ref{equ:a.3}) yields that both $\|\t \phi_1(u)\|=\|\phi_1(u)\|$
and $\|\t \phi_1(v)\|=\|\phi_1(v)\|$. We then have from
(\ref{equ:a.1}) that
$${\|\phi'(u)\| \over \|\phi'(v)\|}={\|\phi_1(u)\| \over \|\phi_1(v)\|},$$
so that, by the induction hypothesis,
$$\min \Delta \le {\|\phi'(u)\| \over \|\phi'(v)\|} \le \max \Delta,$$
as required.

\medskip
\noindent {\bf Case A.2.} $\t$ occurs at least $\|u\|+3$ times in
$\phi_1$.
\medskip

In view of Corollary ~\ref{cor:4.5} ~(i) with $\t$ in place of
$\s$, we have from (\ref{equ:a.3}) both $\|\t
\phi_1(u)\|=\|\phi_1(u)\|$ and $\|\t \phi_1(v)\|=\|\phi_1(v)\|$.
It then follows from (\ref{equ:a.1}) that
$${\|\phi'(u)\| \over \|\phi'(v)\|}={\|\phi_1(u)\| \over \|\phi_1(v)\|},$$
so that, by the induction hypothesis,
$$\min \Delta \le {\|\phi'(u)\| \over \|\phi'(v)\|} \le \max \Delta,$$
as desired. \qed

\medskip
\noindent
{\bf Claim B.}
$$\min \Delta \le {\|\s: \phi': u\| \over \|\s : \phi': v\|}\le \max \Delta$$
\medskip

\noindent {\it Proof of Claim ~B.} As in the proof of Claim A,
writing
$$\phi'=\t \phi_1,$$
where $\phi_1$ is a chain of type (C1), we consider two cases
separately.

\medskip
\noindent {\bf Case B.1.} $\s$ occurs at least $\|u\|+3$ times in
$\phi_1$.
\medskip

In this case, write $$\phi_1=\t^{m-1} \s \phi_2,$$ where $m \ge 1$
and $\phi_2$ is a chain of type (C1). Since $\phi'=\t \phi_1$,
$$\phi'=\t^m \s \phi_2.$$ Then by Lemma ~\ref{lem:4.1} ~(i), we have
\begin{equation}
\begin{aligned} \label{equ:b.1}
\|\s \phi'(u)\|-\|\phi'(u)\|&=\|\s \t^m \phi_2(u)\|-\|\t^m \phi_2(u)\|+m(\|\s \phi_2(u)\|-\|\phi_2(u)\|); \\
\|\s \phi'(v)\|-\|\phi'(v)\|&=\|\s \t^m \phi_2(v)\|-\|\t^m \phi_2(v)\|+m(\|\s \phi_2(v)\|-\|\phi_2(v)\|).
\end{aligned}
\end{equation}
Here, since $\phi_2$ is a chain of type (C1) which contains at
least $\|u\|+2$ factors of $\s$, Corollary ~\ref{cor:4.5} ~(i)
yields that $\|\s \t^m \phi_2(u)\|=\|\t^m \phi_2(u)\|$ if and only
if $\|\s \t^m \phi_2(v)\|=\|\t^m \phi_2(v)\|$. So if $\|\s \t^m
\phi_2(u)\|=\|\t^m \phi_2(u)\|$ or $\|\s \t^m \phi_2(v)\|=\|\t^m
\phi_2(v)\|$, then we get from (\ref{equ:b.1}) that
\[
\begin{aligned}
\|\s \phi'(u)\|-\|\phi'(u)\|&=m(\|\s \phi_2(u)\|-\|\phi_2(u)\|); \\
\|\s \phi'(v)\|-\|\phi'(v)\|&=m(\|\s \phi_2(v)\|-\|\phi_2(v)\|).
\end{aligned}
\]
This gives us
$${\|\s: \phi': u\| \over \|\s : \phi': v\|}=
{\|\s: \phi_2: u\| \over \|\s : \phi_2: v\|},$$ and hence the
desired inequalities
$$\min \Delta \le {\|\s: \phi': u\| \over \|\s : \phi': v\|}
\le \max \Delta$$ follow by the induction hypothesis.

Now let us assume that
\[
\text{ \rm $\|\s \t^m \phi_2(u)\| \neq \|\t^m \phi_2(u)\|$ and
$\|\s \t^m \phi_2(v)\| \neq \|\t^m \phi_2(v)\|$.}
\]
Again by Corollary ~\ref{cor:4.5} ~(i), we have $\|\s
\phi_2(u)\|=\|\phi_2(u)\|$ if and only if $\|\s
\phi_2(v)\|=\|\phi_2(v)\|$. Hence if $\|\s \phi_2(u)\|=\|
\phi_2(u)\|$ or $\|\s \phi_2(v)\|=\|\phi_2(v)\|$, then, from
(\ref{equ:b.1}),
\[
\begin{aligned}
\|\s \phi'(u)\|-\|\phi'(u)\|&=\|\s \t^m \phi_2(u)\|-\|\t^m \phi_2(u)\|; \\
\|\s \phi'(v)\|-\|\phi'(v)\|&=\|\s \t^m \phi_2(v)\|-\|\t^m
\phi_2(v)\|.
\end{aligned}
\]
This yields
$${\|\s: \phi': u\| \over \|\s : \phi': v\|}=
{\|\s: \t^m \phi_2: u\| \over \|\s : \t^m \phi_2: v\|},$$ which
gives us
$$\min \Delta \le {\|\s: \phi': u\| \over \|\s : \phi': v\|}
\le \max \Delta$$ by the induction hypothesis.

So let us further assume that
\[
\text {\rm $\|\s \phi_2(u)\| \neq \|\phi_2(u)\|$ and $\|\s
\phi_2(v)\| \neq \|\phi_2(v)\|$.}
\]
It then follows from (\ref{equ:b.1}) that
\begin{equation}
\begin{aligned} \label{equ:b.4}
\|\s \phi'(u)\|-\|\phi'(u)\|&=\|\s : \t^m \phi_2: u\|+m\|\s : \phi_2: u\|; \\
\|\s \phi'(v)\|-\|\phi'(v)\|&=\|\s : \t^m \phi_2 : v\|+m\|\s :
\phi_2: v\|.
\end{aligned}
\end{equation}
Since
$$\min \Delta \le {\|\s: \t^m \phi_2: u\| \over \|\s : \t^m \phi_2: v\|},
{\|\s: \phi_2: u\| \over \|\s : \phi_2: v\|} \le \max \Delta$$ by
the induction hypothesis, we have from (\ref{equ:b.4}) that
$$\min \Delta \le {\|\s: \phi': u\| \over \|\s : \phi': v\|}
\le \max \Delta,$$ as required.

\medskip
\noindent {\bf Case B.2.} $\t$ occurs at least $\|u\|+3$ times in
$\phi_1$.
\medskip

In this case, it follows from Lemma ~\ref{lem:4.1} ~(ii) with $\s, \t$ interchanged and $m=0$ that
\begin{equation}
\begin{aligned} \label{equ:b.2}
\|\s \phi'(u)\|-\|\phi'(u)\|&=\|\s \phi_1(u)\|-\|\phi_1(u)\|+\|\t \phi_1(u)\|-\|\phi_1(u)\|; \\
\|\s \phi'(v)\|-\|\phi'(v)\|&=\|\s \phi_1(v)\|-\|\phi_1(v)\|+\|\t
\phi_1(v)\|-\|\phi_1(v)\|.
\end{aligned}
\end{equation}
Here, by Corollary ~\ref{cor:4.5} ~(i) with $\t$ in place of $\s$,
we have $\|\t \phi_1(u)\|=\|\phi_1(u)\|$ if and only if $\|\t
\phi_1(v)\|=\|\phi_1(v)\|$. Hence if $\|\t
\phi_1(u)\|=\|\phi_1(u)\|$ or $\|\t \phi_1(v)\|=\|\phi_1(v)\|$,
then, by (\ref{equ:b.2}),
\[
\begin{aligned}
\|\s \phi'(u)\|-\|\phi'(u)\|&=\|\s \phi_1(u)\|-\|\phi_1(u)\|; \\
\|\s \phi'(v)\|-\|\phi'(v)\|&=\|\s \phi_1(v)\|-\|\phi_1(v)\|,
\end{aligned}
\]
and thus
\[
\begin{aligned}
\|\s : \phi' : u \|&=\|\s : \phi_1 : u \|;\\
\|\s : \phi' : v \|&=\|\s : \phi_1 : v \|.
\end{aligned}
\]
Then by the induction hypothesis,
$$\min \Delta \le
{\|\s: \phi': u\| \over \|\s : \phi': v\|} \le \max \Delta,$$ as
desired.

Now assume that
\[
\text{\rm $\|\t \phi_1(u)\| \neq \|\phi_1(u)\|$ and $\|\t
\phi_1(v)\| \neq \|\phi_1(v)\|$.}
\]
We shall show that $\|\s \phi_1(u)\|=\|\phi_1(u)\|$ if and only if
$\|\s \phi_1(v)\|=\|\phi_1(v)\|$. Let $\|\s
\phi_1(u)\|=\|\phi_1(u)\|$. If $\phi_1$ ends in $\s$, then, by
Corollary ~\ref{cor:4.5} ~(iii) with $\s, \t$ interchanged, we
have $\|\s \phi_1(v)\|=\|\phi_1(v)\|$. On the other hand, if
$\phi_1$ ends in $\t$, then, by Corollary ~\ref{cor:4.5} ~(ii)
with $\s, \t$ interchanged, we get $\|\t
\phi_2(u)\|=\|\phi_2(u)\|$, where $\phi_1=\t \phi_2$. But then
from Lemma ~\ref{lem:4.3.3} ~(i) with $\s$, $\t$ interchanged, it
follows that $\|\t^2 \phi_2(u)\| = \|\t \phi_2(u)\|$, namely,
$\|\t \phi_1(u)\| = \|\phi_1(u)\|$, which contradicts our
assumption $\|\t \phi_1(u)\| \neq \|\phi_1(u)\|$. Therefore, we
must have $\|\s \phi_1(v)\|=\|\phi_1(v)\|$. Conversely, if $\|\s
\phi_1(v)\|=\|\phi_1(v)\|$, then, for a similar reason, it must
follow that $\|\s \phi_1(u)\|=\|\phi_1(u)\|$.

Thus if $\|\s \phi_1(u)\|=\|\phi_1(u)\|$ or $\|\s
\phi_1(v)\|=\|\phi_1(v)\|$, then, from (\ref{equ:b.2}),
\[
\begin{aligned}
\|\s \phi'(u)\|-\|\phi'(u)\|&=\|\t \phi_1(u)\|-\|\phi_1(u)\|; \\
\|\s \phi'(v)\|-\|\phi'(v)\|&=\|\t \phi_1(v)\|-\|\phi_1(v)\|,
\end{aligned}
\]
and so
\[
\begin{aligned}
\|\s : \phi' : u \|&=\|\t : \phi_1 : u \|;\\
\|\s : \phi' : v \|&=\|\t : \phi_1 : v \|.
\end{aligned}
\]
Then by the induction hypothesis,
$$\min \Delta \le
{\|\s: \phi': u\| \over \|\s : \phi': v\|} \le \max \Delta,$$ as
required.

So assume further that
\[
\text{\rm $\|\s \phi_1(u)\| \neq \|\phi_1(u)\|$ and $\|\s
\phi_1(v)\| \neq \|\phi_1(v)\|$.}
\]
It follows from (\ref{equ:b.2}) that
\begin{equation}
\begin{aligned} \label{equ:b.3}
\|\s \phi'(u)\|-\|\phi'(u)\|&=\|\s : \phi_1: u\|+\|\t : \phi_1: u\|; \\
\|\s \phi'(v)\|-\|\phi'(v)\|&=\|\s :\phi_1 :v\|+\|\t :\phi_1: v\|.
\end{aligned}
\end{equation}
Since
$$\min \Delta \le  {\|\t: \phi_1: u\| \over \|\t : \phi_1: v\|},
{\|\s: \phi_1: u\| \over \|\s : \phi_1: v\|} \le \max \Delta$$ by
the induction hypothesis, we obtain from (\ref{equ:b.3}) that
$$\min \Delta \le
{\|\s: \phi': u\| \over \|\s : \phi': v\|} \le \max \Delta,$$ as
desired. \qed

\medskip
\noindent
{\bf Claim C.}
$$\min \Delta \le
{\|\t: \phi': u\| \over \|\t : \phi': v\|} \le \max \Delta$$
\medskip

\noindent {\it Proof of Claim C.} As in the proof of Claims A and
B, writing
$$\phi'=\t \phi_1,$$
where $\phi_1$ is a chain of type (C1), we consider two cases
separately.

\medskip
\noindent {\bf Case C.1.} $\s$ occurs at least $\|u\|+3$ times in
$\phi_1$.
\medskip

As in Case B.1, write $$\phi_1=\t^{m-1} \s \phi_2,$$ where $m \ge
1$ and $\phi_2$ is a chain of type (C1). Since $\phi'=\t
\phi_1$,
$$\phi'=\t^m \s \phi_2.$$
It then follows from Lemma ~\ref{lem:4.1} ~(ii)
that
\begin{equation}
\begin{aligned}
\|\t \phi'(u)\|-\|\phi'(u)\|&=\|\t \t^m \phi_2(u)\|-\|\t^m \phi_2(u)\|+\|\s \phi_2(u)\|-\|\phi_2(u)\|; \\
\|\t \phi'(v)\|-\|\phi'(v)\|&=\|\t \t^m \phi_2(v)\|-\|\t^m \phi_2(v)\|+\|\s \phi_2(v)\|-\|\phi_2(v)\|.
\end{aligned}
\end{equation}
By Corollary ~\ref{cor:4.5} ~(i), we have $\|\s
\phi_2(u)\|=\|\phi_2(u)\|$ if and only if $\|\s
\phi_2(v)\|=\|\phi_2(v)\|$. Also by Corollary ~\ref{cor:4.5}
~(iii), we get $\|\t \t^m \phi_2(u)\|=\|\t^m \phi_2(u)\|$ if and
only if $\|\t \t^m \phi_2(v)\|=\|\t^m \phi_2(v)\|$. Hence we can
apply a similar argument as in Cases ~B.1 and B.2 to obtain the
desired inequalities
$$\min \Delta \le
{\|\t: \phi': u\| \over \|\t : \phi': v\|} \le \max \Delta.$$

\medskip
\noindent {\bf Case C.2.} $\t$ occurs at least $\|u\|+3$ times in
$\phi_1$.
\medskip

By Lemma ~\ref{lem:4.1} ~(i) with $\s$, $\t$ interchanged and
$m=0$, we have
\[
\begin{aligned}
\|\t \phi'(u)\|-\|\phi'(u)\|&=\|\t \phi_1(u)\|-\|\phi_1(u)\|; \\
\|\t \phi'(v)\|-\|\phi'(v)\|&=\|\t \phi_1(v)\|-\|\phi_1(v)\|.
\end{aligned}
\]
It then follows that
\[
\begin{aligned}
\|\t : \phi' : u \|&=\|\t : \phi_1 : u \|;\\
\|\t : \phi' : v \|&=\|\t : \phi_1 : v \|,
\end{aligned}
\]
so that
$$\min \Delta \le
{\|\t: \phi': u\| \over \|\t : \phi': v\|} \le \max \Delta$$ by
the induction hypothesis. This completes the proof of Claim ~C.
\qed

Now the theorem is completely proved.
\end{proof}

The following theorem is the converse of Theorem ~\ref{thm:4.7}.

\begin{theorem} \label{thm:4.8} Let $u, v \in F_2$ with $\|u\| \ge \|v\|$, and
$\Omega, \Omega_1$ and $\Omega_2$ be defined as in the statement
of Theorem ~\ref{thm:4.7}. Put $k=\|u\|+1$. Suppose that $u$ and
$v$ are boundedly translation equivalent in $F_2$. Then
\[
\begin{aligned}
&\|\s^{k+1} \psi_1(u)\|=\|\s^k \psi_1(u)\| \ \text{if and only if}
\ \|\s^{k+1} \psi_1(v)\|=\|\s^k \psi_1(v)\|; \\
&\|\t^{k+1} \psi_1(u)\|=\|\t^k \psi_1(u)\| \ \text{if and only if}
\ \|\t^{k+1} \psi_1(v)\|=\|\t^k \psi_1(v)\|,
\end{aligned}
\]
for every $\psi_1 \in \Omega_1$, and
\[
\begin{aligned}
&\|\s^{-k-1} \psi_2(u)\|=\|\s^{-k} \psi_2(u)\| \ \text{if and only
if} \ \|\s^{-k-1} \psi_2(v)\|=\|\s^{-k} \psi_2(v)\|; \\
&\|\t^{-k-1} \psi_2(u)\|=\|\t^{-k} \psi_2(u)\| \ \text{if and only
if} \ \|\t^{-k-1} \psi_2(v)\|=\|\t^{-k} \psi_2(v)\|,
\end{aligned}
\]
for every $\psi_2 \in \Omega_2$.
\end{theorem}

\begin{proof} Suppose on the contrary that
\begin{equation} \label{equ:48.1}
\|\s^{k+1} \psi_1(u)\|=\|\s^k \psi_1(u)\| \ \text{but} \
\|\s^{k+1} \psi_1(v)\| \neq \|\s^k \psi_1(v)\|
\end{equation}
for some $\psi_1 \in \Omega_1$. (The treatment of the other cases
is similar.) Put
\[
K=\|\s^{k+1} \psi_1(v)\|-\|\s^k \psi_1(v)\|.
\]
By Lemma ~\ref{lem:4.3} ~(i) and the second inequality of
(\ref{equ:48.1}), we have $K \ge 1$. By repeatedly applying Lemma
~\ref{lem:4.1} ~(i), we deduce that
\[
\begin{aligned}
&\|\s^{i+1} \psi_1(u)\|=\|\s^k \psi_1(u)\| \ \text{for every} \ i
\ge k;\\
&\|\s^{i+1} \psi_1(v)\|=\|\s^k \psi_1(v)\|+K(i+1-k) \ \text{for
every} \ i \ge k.
\end{aligned}
\]
Hence
\[
{\|\s^{i+1} \psi_1(u)\| \over \|\s^{i+1} \psi_1(v)\|}={\|\s^k
\psi_1(u)\| \over \|\s^k \psi_1(v)\|+K(i+1-k)}
\]
for every $i \ge k$, and thus
\[
\lim_{i \to \infty} {\|\s^{i+1} \psi_1(u)\| \over \|\s^{i+1}
\psi_1(v)\|}=0.
\]
This contradiction to the the hypothesis that $u$ and $v$ are
boundedly translation equivalent in $F_2$ completes the proof.
\end{proof}

Consequently, in view of Theorems ~\ref{thm:4.7} and
~\ref{thm:4.8}, we obtain the following algorithm to determine
bounded translation equivalence in $F_2$.

\begin{algorithm} Let $u, v \in F_2$ with $\|u\| \ge \|v\|$, and
let $\Omega, \Omega_1$ and $\Omega_2$ be defined as in the
statement of Theorem ~\ref{thm:4.7}. Put $k=\|u\|+1$. Check if it
is true that
\[
\begin{aligned}
&\|\s^{k+1} \psi_1(u)\|=\|\s^k \psi_1(u)\| \ \text{if and only if}
\ \|\s^{k+1} \psi_1(v)\|=\|\s^k \psi_1(v)\|; \\
&\|\t^{k+1} \psi_1(u)\|=\|\t^k \psi_1(u)\| \ \text{if and only if}
\ \|\t^{k+1} \psi_1(v)\|=\|\t^k \psi_1(v)\|,
\end{aligned}
\]
for each $\psi_1 \in \Omega_1$, and if it is true that
\[
\begin{aligned}
\|\s^{-k-1} \psi_2(u)\|=\|\s^{-k} \psi_2(u)\| \ \text{if and only
if} \ \|\s^{-k-1} \psi_2(v)\|=\|\s^{-k} \psi_2(v)\|;\\
\|\t^{-k-1} \psi_2(u)\|=\|\t^{-k} \psi_2(u)\| \ \text{if and only
if} \ \|\t^{-k-1} \psi_2(v)\|=\|\t^{-k} \psi_2(v)\|,
\end{aligned}
\]
for each $\psi_2 \in \Omega_2$. If so, conclude that $u$ and $v$
are boundedly translation equivalent in $F_2$; otherwise conclude
that $u$ and $v$ are not boundedly translation equivalent in
$F_2$.
\end{algorithm}

\section{Fixed point groups of automorphisms of $F_2$}

In this section, we shall demonstrate that there exists an
algorithm to decide whether or not a given finitely generated
subgroup of $F_2$ is the fixed point group of some automorphism of
$F_2$. If $H=\langle u_1, \dots, u_k \rangle$ is a finitely
generated subgroup of $F_2$, then we define
$$|H|:=\max_{1 \le i \le k} |u_i|.$$
Clearly $\|u_i\| \le |u_i| \le |H|$ for every $i=1, \dots, k$.

\begin{theorem} \label{pro:5.3} Let $H=\langle u_1, \dots, u_k \rangle$ be a finitely generated
subgroup of $F_2$. Suppose that $\phi$ is a chain of type (C1)
with $|\phi| \ge 4|H|+5$ such that $\|\phi(u_i)\|=\|u_i\|$ for
every $i=1, \dots, k$. Then there exists a chain $\psi$ of type
(C1) with $|\psi|<|\phi|$ such that $[\psi(u_i)]=[\phi(u_i)]$ for
every $i=1, \dots, k$.
\end{theorem}

\begin{proof} Since $\phi$ is a chain of type (C1)
with $|\phi| \ge 4|H|+5$, $\phi$ contains at least $2|H|+3$
factors of $\s$ or $\t$. Suppose that $\phi$ contains at least
$2|H|+3$ factors of $\s$ (the other case is similar). We may write
\begin{equation} \label{equ:5.0}
\phi=\t^{m_t} \s^{\ell_t}  \cdots \t^{m_1} \s^{\ell_1} \phi',
\end{equation}
where all $\ell_i, m_i > 0$ but $\ell_1$ and $m_t$ may be zero,
and $\phi'$ is a chain of type (C1) which contains exactly $|H|+2$
factors of $\s$.

Suppose that there exists $u_j$ ($1 \le j \le k$) such that $\| \s
\phi'(u_j)\| \neq \|\phi'(u_j)\|$. Put
$$K=\| \s \phi'(u_j)\|-\|\phi'(u_j)\|.$$
Since $\phi'$ contains at least $\|u_j\|+2$ factors of $\s$, by
Lemma \ref{lem:4.3} ~(i), $K \ge 1$. Furthermore, since $\phi$
contains at least $2|H|+3$ factors of $\s$ and $\phi'$ contains
exactly $|H|+2$ factors of $\s$,
\begin{equation} \label{equ:5.1.7}
\sum_{i=1}^t \ell_i \ge |H|+1 \ge \|u_j\|+1.
\end{equation}

From the following claim, we shall obtain a contradiction.

\medskip
\noindent {\bf Claim.} $\|\phi(u_j)\|-\|\phi'(u_j)\| \ge
\|u_j\|+1$.
\medskip

\noindent {\it Proof of the Claim.} First assume that $m_1=0$ in
(\ref{equ:5.0}). Then $\phi=\s^{\ell_1}\phi'$, and so, from
(\ref{equ:5.1.7}), $\ell_1 \ge \|u_j\|+1$. By repeatedly applying
Lemma ~\ref{lem:4.1} ~(i), we have
$$\|\phi(u_j)\|-\|\phi'(u_j)\|=\ell_1K.$$
Since $K \ge 1$, it follows that
$$\|\phi(u_j)\|-\|\phi'(u_j)\|\ge \ell_1 \ge \|u_j\|+1,$$ as desired.

Next assume that $m_1 >0$ in (\ref{equ:5.0}). In view of Lemmas
~\ref{lem:4.1} and \ref{lem:4.3}, we can observe that
\[
\begin{aligned}
\|\s^{\ell_1} \phi'(u_j)\|-\|\phi'(u_j)\|&=\ell_1K; \\
\|\t^{m_1} \s^{\ell_1} \phi'(u_j)\|-\|\s^{\ell_1} \phi'(u_j)\|
&\ge m_1K; \\
&\cdots  \\
\|\s^{\ell_t}  \cdots \t^{m_1} \s^{\ell_1}
\phi'(u_j)\|-\|\t^{m_{t-1}}\cdots \t^{m_1}
\s^{\ell_1}\phi'(u_j)\| &\ge \ell_tK; \\
\|\t^{m_t} \s^{\ell_t}  \cdots \t^{m_1} \s^{\ell_1}
\phi'(u_j)\|-\|\s^{\ell_t} \cdots \t^{m_1} \s^{\ell_1}\phi'(u_j)\|
&\ge m_tK.
\end{aligned}
\]
Summing up all of these inequalities together with
(\ref{equ:5.1.7}) yields
\[
\begin{aligned}
\|\phi(u_j)\|-\|\phi'(u_j)\| &\ge \sum_{i=1}^t(\ell_i+m_i) K \\
 &\ge (\sum_{i=1}^t \ell_i) K \\
 & \ge \sum_{i=1}^t \ell_i \\
 & \ge \|u_j\|+1,
\end{aligned}
\]
as required. This completes the proof of the claim. \qed

It then follows from the claim that
$$\|\phi(u_j)\| \ge \|\phi'(u_j)\| + \|u_j\|+1 \ge \|u_j\|+1.$$
But this yields a contradiction to the hypothesis that
$\|\phi(u_j)\|=\|u_j\|$. Therefore, we must have $\| \s
\phi'(u_i)\|=\|\phi'(u_i)\|$ for every $i=1, \dots, k$. Then for
each $i=1, \dots, k$,
\begin{equation} \label{equ:5.1.8}
\begin{split}
0=\|\s
\phi'(u_i)\|-\|\phi'(u_i)\|&=n([\phi'(u_i)];a)-2n([\phi'(u_i);
a,b^{-1})
\\
&=n([\phi'(u_i)];a, a)+n([\phi'(u_i)];a, b)-n([\phi'(u_i)];a,
b^{-1}).
\end{split}
\end{equation}
Here, since $\phi'$ contains at least $\|u_i\|+2$ factors of $\s$,
by Lemma ~\ref{lem:2.4}, there cannot occur proper cancellation in
passing from $[\phi'(u_i)]$ to $[\s \phi'(u_i)]$, and so every
subword of $[\phi'(u_i)]$ of the form $ab^{-1}$ or $ba^{-1}$ is
necessarily part of a subword of the form $ab^{-r}a^{-1}$ or
$ab^ra^{-1}$ ($r
>0$), respectively. This implies that
\[
n([\phi'(u_i)];a, b) \ge n([\phi'(u_i)];a, b^{-1}),
\]
so that, from (\ref{equ:5.1.8}),
\begin{equation} \label{equ:5.1.9}
n([\phi'(u_i)];a, b) = n([\phi'(u_i)];a, b^{-1}) \ \text{and} \
n([\phi'(u_i)];a, a)=0.
\end{equation}
From the fact that no proper cancellation can occur in passing
from $[\phi'(u_i)]$ to $[\s \phi'(u_i)]$ together with
(\ref{equ:5.1.9}), each cyclic word $[\phi'(u_i)]$ must have the
form
$$
[\phi'(u_i)]=[b^{s_{i1}}ab^{t_{i1}}a^{-1} \cdots
b^{s_{ir}}ab^{t_{ir}}a^{-1}],
$$
where every $s_{ij}, t_{ij}$ is a nonzero integer, and hence
$$[\s \phi'(u_i)]=[\phi'(u_i)]$$ for every $i=1, \dots, t$.

Thus letting
$$\psi=\t^{m_t} \s^{\ell_t}  \cdots \t^{m_1} \s^{\ell_1-1}
\phi',$$ we finally have
$$[\psi(u_i)]=[\phi(u_i)]$$ for every $i=1, \dots, t$. Obviously $|\psi|<|\phi|$,
and so the proof of the theorem is completed.
\end{proof}

We remark that Theorem \ref{pro:5.3} also holds if (C1) is
replaced by (C2). From now on, let
\[
\d_1=(\{a^{\pm 1}\}, b), \quad \d_2=(\{a^{\pm 1}\}, b^{-1}), \quad
\d_3=(\{b^{\pm 1}\}, a),  \quad \d_4=(\{b^{\pm 1}\}, a^{-1})
\]
be Whitehead automorphisms of $F_2$ of type (W2).

\begin{lemma} \label{lem:5.0.5} Let $\a$ be a Whitehead automorphism of
$F_2$ of type (W2). Then $\a$ can be expressed as a composition of
$\s^{\pm 1}$, $\t^{\pm 1}$ and $\d_i$'s.
\end{lemma}

\begin{proof} If $\a$ is not one of $\s^{\pm 1}$, $\t^{\pm 1}$ and $\d_i$'s,
then $\a$ must be one of $(\{a^{-1}\}, b)$, $(\{a^{-1}\},
b^{-1})$, $(\{b^{-1}\}, a)$ and $(\{b^{-1}\}, a^{-1})$. Then the
following easy identities
\[
\begin{aligned}
(\{a^{-1}\}, b)=\d_1 \s^{-1}; &\qquad (\{a^{-1}\}, b^{-1})=\d_2
\s;
\\
(\{b^{-1}\}, a)=\d_3 \t^{-1}; &\qquad (\{b^{-1}\}, a^{-1})=\d_4
\t
\end{aligned}
\]
imply the required result.
\end{proof}

The following two technical lemmas can be easily proved by direct
calculations.

\begin{lemma} \label{lem:5.1} The following identities hold.
\[
\begin{aligned}
\s \d_1 &= \d_1 \s; & \s \d_2 &= \d_2 \s; & \s \d_3 &= \d_1
\d_3 \s; & \s \d_4 &= \d_4 \d_2 \s;\\
\t \d_1 &= \d_3 \d_1 \t; & \t \d_2 &= \d_2 \d_4 \t; & \t \d_3
&= \d_3 \t; & \t \d_4 &= \d_4 \t;\\
\s^{-1} \d_1 &= \d_1 \s^{-1}; & \s^{-1} \d_2 &= \d_2 \s^{-1}; &
\s^{-1} \d_3 &= \d_2
\d_3 \s^{-1}; & \s^{-1} \d_4 &= \d_4 \d_1 \s^{-1};\\
\t^{-1} \d_1 &= \d_4 \d_1 \t^{-1}; & \t^{-1} \d_2 &= \d_2 \d_3
\t^{-1}; & \t^{-1} \d_3 &= \d_3 \t^{-1}; & \t^{-1} \d_4 &= \d_4
\t^{-1}.
\end{aligned}
\]
\end{lemma}

\begin{lemma} \label{lem:5.1.5} The following identities hold.
\[
\begin{aligned}
\s \t^{-1}&= \pi \d_1 \s^{-1}; & \s^{-1} \t &= \pi^{-1} \d_3 \s; &
\t \s^{-1} &= \pi^{-1} \d_3 \t^{-1}; & \t^{-1} \s &= \pi
\d_1 \t;\\
\s \pi &=\pi \d_3 \t^{-1}; & \s \pi^{-1}&=\pi^{-1} \t^{-1}; &
\s^{-1} \pi &=\pi \d_4 \t; & \s^{-1} \pi^{-1} &=\pi^{-1} \t; \\
\t \pi &=\pi \s^{-1}; & \t \pi^{-1} &=\pi^{-1} \d_1 \s^{-1}; &
\t^{-1} \pi &=\pi \s; & \t^{-1} \pi^{-1} &=\pi^{-1} \d_2 \s,
\end{aligned}
\]
where $\pi$ is a Whitehead automorphism of $F_2$ of type (W1) that
sends $a$ to $b$ and $b$ to $a^{-1}$.
\end{lemma}

The following corollary gives a nice description of automorphisms
of $F_2$.

\begin{corollary} \label{cor:5.2} Every automorphism $\phi$ of $F_2$
can be represented as
\[
\phi = \b \d \phi',
\]
where $\b$ is a Whitehead automorphism of $F_2$ of type (W1), $\d$
is a composition of $\d_i$'s, and $\phi'$ is a chain of type (C1)
or (C2).
\end{corollary}

\begin{proof} By Whitehead's Theorem (cf. \cite{whitehead}) together with
Lemmas ~\ref{lem:5.0.5} and ~\ref{lem:5.1}, an automorphism $\phi$
of $F_2$ can be expressed as
\begin{equation} \label{equ:5.2}
\phi = \b' \d' \t^{q_t}\s^{p_t} \cdots \t^{q_1}\s^{p_1},
\end{equation}
where $\b'$ is a Whitehead automorphism of $F_2$ of type (W1),
$\d'$ is a composition of $\d_i$'s, and both $p_j, \, q_j$ are
(not necessarily positive) integers for every $j=1, \dots, t$. If
not every $p_j$ and $q_j$ has the same sign (including $0$), apply
repeatedly Lemma ~\ref{lem:5.1.5} to the chain on the right-hand
side of (\ref{equ:5.2}) to obtain that either $\phi = \b'\pi^r \d
\t^{m_k}\s^{l_k} \cdots \t^{m_1}\s^{l_1}$ or $\phi = \b'\pi^r \d
\t^{-m_k}\s^{-l_k} \cdots \t^{-m_1}\s^{-l_1}$, where $\pi$ is as
in Lemma ~\ref{lem:5.1.5}, $r \in \Bbb Z$, $\d$ is a composition
of $\d_i$'s, and both $l_j, \, m_j \ge 0$ for every $j=1, \dots,
k$. Putting $\b=\b'\pi^r$, we obtain the required result.

\end{proof}

The following is the main result of this section.

\begin{theorem} \label{pro:5.4} Let $H=\langle u_1, \dots, u_k \rangle$
be a finitely generated subgroup of $F_2$. Suppose that $H$ is the
fixed point group of an automorphism $\phi$ of $F_2$. Let
$\Omega_1$ be the set of all chains of type (C1) or (C2) of length
less than or equal to $4|H|+4$, and let $\Omega_2$ be the set of
all compositions of $\d_i$'s of length less than or equal to
$(2^{4|H|+4}+1)|H|$. Put
$$\Omega=\{ \b \d' \psi' \, | \, \psi' \in \Omega_1, \d' \in
\Omega_2, \text{\it and $\b$ is a Whitehead auto of $F_2$ of type
(W1)}\}.$$ Then there exists $\psi \in \Omega$ of which $H$ is the
fixed point group.
\end{theorem}

\begin{proof} By Corollary ~\ref{cor:5.2}, $\phi$ can be written as
$$\phi=\b \d \phi',$$
where $\b, \d$ and $\phi'$ are indicated as in the statement of
Corollary ~\ref{cor:5.2}.

Since $\phi(u_i)=u_i$ for every $i=1, \dots, k$, it is easy to see
that
$$\|\phi'(u_i)\|=\|u_i\|$$
for every $i=1, \dots, k$. Then apply Theorem ~\ref{pro:5.3}
continuously to obtain $\psi' \in \Omega_1$ such that
$$[\psi'(u_i)]=[\phi'(u_i)]$$
for every $i=1, \dots, k$. Since $|\d
\phi'(u_i)|=|\phi(u_i)|=|u_i| \le |H|$ and $|\psi'(u_i)| \le
2^{4|H|+4}|u_i| \le 2^{4|H|+4}|H|$ for every $i=1, \dots, k$, we
must have $\d' \in \Omega_2$ such that
$$\d' \psi' (u_i)=\d \phi'(u_i)$$
for every $i=1, \dots, k$, and hence
$$\b \d' \psi' (u_i)=\b \d \phi' (u_i)=u_i$$
for every $i=1, \dots, k$. Therefore, letting $$\psi=\b \d'
\psi',$$ we finally have $\psi \in \Omega$ and that $H$ is the
fixed point subgroup of $\psi$. This completes the proof of the
theorem.
\end{proof}

In conclusion, we naturally derive from Theorem ~\ref{pro:5.4} the
following algorithm to decide whether or not a given finitely
generated subgroup of $F_2$ is the fixed point group of some
automorphism of $F_2$.

\begin{algorithm} Let $H=\langle u_1, \dots, u_k \rangle$ be a finitely generated
subgroup of $F_2$. Let $\Omega_1, \Omega_2$ and $\Omega$ be
defined as in the statement of Theorem ~\ref{pro:5.4}. Clearly
$\Omega$ is a finite set. Check if there is $\psi \in \Omega$ for
which $\psi(u_i)=u_i$ holds for every $i=1, \dots, k$. If so,
conclude that $H$ is the fixed point group of some automorphism of
$F_2$; otherwise conclude that $H$ is not the fixed point group of
any automorphism of $F_2$.
\end{algorithm}


\bibstyle{plain}

\end{document}